\documentclass[12pt, reqno]{amsart}
\usepackage{amsmath, amsthm, amscd, amsfonts, amssymb, graphicx, color, mathrsfs}
\usepackage[bookmarksnumbered, colorlinks, plainpages]{hyperref}

\newtheorem{theorem}{Theorem}[section]
\newtheorem{lemma}[theorem]{Lemma}

\newtheorem{corollary}[theorem]{Corollary}
\theoremstyle{definition}
\newtheorem{definition}[theorem]{Definition}

\theoremstyle{remark}
\newtheorem{remark}[theorem]{Remark}
\numberwithin{equation}{section}

\begin{document}
\setcounter{page}{1}

\title[Besov continuity for global  operators: the critical case $p=q=\infty.$ ]{  Besov continuity for global  operators on compact Lie groups: the critical case $p=q=\infty.$}

\author[D. Cardona]{Duv\'an Cardona}
\address{
  Duv\'an Cardona:
  \endgraf
  Department of Mathematics  
  \endgraf
  Pontificia Universidad Javeriana.
  \endgraf
  Bogot\'a
  \endgraf
  Colombia
  \endgraf
  {\it E-mail address} {\rm duvanc306@gmail.com;
cardonaduvan@javeriana.edu.co}
  }


\begin{abstract}
In this note, we study the mapping properties of  global pseudo-differential operators with symbols in Ruzhansky-Turunen classes  on Besov spaces $B^{s}_{\infty,\infty}(G).$  The considered classes satisfy Fefferman type conditions of limited regularity.
\end{abstract} \maketitle

\tableofcontents
\section{Introduction}
This note on the Besov boundedness   of pseudo-differential operators on compact Lie groups in $B^s_{\infty,\infty}$ is based in the matrix-valued quantization procedure developed by M. Ruzhansky and V. Turunen in \cite{Ruz}. 

The Besov spaces $B^s_{p,q}$ arose from attempts to unify the various definitions of
several fractional-order Sobolev spaces. By following the historical note of Grafakos \cite[p. 113]{GrafakosBook}, we recall that Taibleson studied the generalized H\"older-Lipchitz spaces $\Lambda^{s}_{p,q}$  on $\mathbb{R}^n,$  and these spaces were named after Besov spaces in honor to O. V.  Besov who obtained a trace theorem and important embedding properties for them (see  Besov \cite{Besov1,Besov2}).  Dyadic decompositions for Besov spaces on $\mathbb{R}^n$ were introduced by J. Peetre as well as other embedding properties (see Peetre \cite{Peetre1,Peetre2}). We will use the formulation of Besov spaces $B^s_{p,q}(G),$ trough of the representation theory of compact Lie groups $G$ introduced and consistently developed by  E. Nursultanov, M. Ruzhansky, and S. Tikhonov  in \cite{RuzBesov,RuzBesov2}. 

 The present paper is a continuation of a series of our previous papers \cite{Cardona,Cardona2,Cardona22,Cardona3}, where  were investigated the mapping properties of global operators (i.e., global pseudo-differential operators on compact Lie groups) on Besov spaces $B^s_{p,q},$ $-\infty<s<\infty,$ $1<p<\infty$ and $0<q\leq \infty.$ So, in this note we study the mapping properties for global operators in the case of Besov spaces $B^{s}_{\infty,\infty},$ $-\infty<s<\infty.$ 
 
The main tool in the proof of the  Besov boundedness results presented  in \cite{Cardona,Cardona2,Cardona22,Cardona3}, are the $L^p$-multipliers theorems proved in Ruzhansky and Wirth \cite{RW3,RW1}, Delgado and Ruzhansky \cite{DRlp}, Fischer \cite{Fischer2} and Akylzhanov and Ruzhansky \cite{RR3}. In general, such  $L^p$-estimates, $1<p<\infty,$  cannot be  extended to $L^{\infty}(G)$ and consequently we need to consider other techniques for the formulation of a boundedness result on $B^s_{\infty,\infty}$. Let us recall some  $L^p$ and  Besov estimates for global operators, in order to announce our main theorem. First of all, we recall some notions  on the global analysis of pseudo-differential operators. 

If $G$ is a compact Lie group and $\widehat{G}$ is its unitary dual, that is the set of  equivalence classes of continuous irreducible unitary representations of $G$, Ruzhansky-Turunen's approach associates to  every bounded linear operator $A$ on $C^\infty(G)$, a matrix valued symbol $\sigma_{A}(x,\xi)$ given by $\sigma_A(x,\xi):=\xi(x)^*(A\xi)(x)$, $x\in G$ and $\xi\in [\xi]\in \widehat{G}$. This  allows us to write the operator $A$ in terms of representations in $G$ as
\begin{equation}\label{operatordefinition'}
Af(x)=\sum_{[\xi]\in \widehat{G}}d_{\xi}\text{Tr}(\xi(x)\sigma_{A}(x,\xi)\widehat{f}(\xi)),
\end{equation}
for all $f\in C^{\infty}(G)$, where $\mathscr{F}(f):=\widehat{f}$ is the Fourier transform on the group $G$.

The H\"ormander classes $\Psi^{m}_{\rho,\delta}(G),$  $m\in\mathbb{R},$ $\rho>\max\{\delta,1-\delta\},$ where characterized in \cite{Ruz,Fischer} by the following condition: $A\in \Psi^{m}_{\rho,\delta}(G)$ if only if its matrix-valued symbol $\sigma_{A}(x,\xi)$ satisfies the inequalities
\begin{equation}\label{CLASSES'}
\Vert \partial_{x}^{\alpha}\Delta^{\beta}\sigma_{A}(x,\xi)\Vert_{op} \leq C_{\alpha,\beta} \langle \xi\rangle^{m-\rho|\beta|+\delta|\alpha|},
\end{equation}
for every $\alpha,\beta\in \mathbb{N}^n.$  The discrete differential operator $\Delta^\beta$ (called difference operator of first order) is the main tool in this theory (see \cite{Ruz,Fischer}).\\
 The $L^p-$mapping properties for global operators on compact Lie groups can be summarized as follows: if $G$ is a compact Lie group and  $n$ is its dimension,   $\varkappa$ is the less integer larger that $\frac{n}{2}$ and  $l:=[n/p]+1,$ under one of the following conditions 
\begin{itemize}
\item $\Vert \partial_{x}^\beta \mathbb{D}^{\alpha}_\xi\sigma_A(x,\xi)\Vert_{op}\leq C_{\alpha}\langle \xi \rangle^{-|\alpha|}, \text{ for all } \,|\alpha|\leq \varkappa,\,\,\,\,\,|\beta|\leq l\text{ y }[\xi]\in\widehat{G}$, (Ruzhansky and Wirth \cite{RW3,RW1}),
\item $\Vert \mathbb{D}_{\xi}^{\alpha}\partial_{x}^{\beta}\sigma_A(x,\xi) \Vert_{op}\leq C_{\alpha,\beta}\langle \xi\rangle^{-m-\rho|\alpha|+\delta|\beta|}, \,\,|\alpha|\leq \varkappa,|\beta|\leq l,\,\, m\geq \varkappa(1-\rho)|\frac{1}{p}-\frac{1}{2}|+\delta l,$ (Delgado and Ruzhansky \cite{DRlp}),
\item $\Vert \mathbb{D}_{\xi}^{\alpha}\partial_{x}^{\beta}\sigma_A(x,\xi) \Vert_{op}\leq C_{\alpha,\beta}\langle \xi\rangle^{-\nu-\rho|\alpha|+\delta|\beta|}, \,\,\alpha\in\mathbb{N}^{n},|\beta|\leq l,$  $0\leq \nu<\frac{n}{2}(1-\rho),$ $|\frac{1}{p}-\frac{1}{2}|\leq \frac{\nu}{n}(1-\rho)^{-1},$ (Delgado and Ruzhansky \cite{DRlp}),
\item  $ \Vert \sigma_{A}\Vert_{\Sigma_{s}}:=\sup_{  [\xi]\in\widehat{G}}[\Vert \sigma_A(\xi)\Vert_{op}+\Vert \sigma_A(\xi)\eta(r^{-2}\mathcal{L}_{G})\Vert_{\dot{H}^{s}(\widehat{G})}]<\infty,$ (Fischer\cite{Fischer2}),
\end{itemize}
the global operator $A\equiv T_a$ is bounded on  $L^{p}(G).$ On the other hand,  if $A:C^{\infty}(G)\rightarrow C^\infty(G)$ is a linear and bounded operator, then under any one of the following conditions
\begin{itemize}
\item $ \Vert  \mathbb{D}_{\xi}^{\alpha}\partial_{x}^\beta\sigma_{A}(x,\xi)\Vert_{op}\leq C_{\alpha}\langle \xi \rangle^{-|\alpha|}, \text{ for all } \,\,\,\,|\alpha|\leq \varkappa,\,\,\,\,\,|\beta|\leq l\text{ and  }[\xi]\in\widehat{G},$
\item $\Vert \mathbb{D}_{\xi}^{\alpha}\partial_{x}^{\beta}\sigma_A(x,\xi) \Vert_{op}\leq C_{\alpha,\beta}\langle \xi\rangle^{-\nu-\rho|\alpha|}, \,\,\alpha\in\mathbb{N}^{n},|\beta|\leq l,\,\,\,|\frac{1}{p}-\frac{1}{2}|\leq \frac{\nu}{n}(1-\rho)^{-1},\,\,0\leq \nu<\frac{n}{2}(1-\rho),$
\item $\Vert \mathbb{D}_{\xi}^{\alpha}\partial_{x}^{\beta}\sigma_A(x,\xi)\Vert_{op}\leq C_{\alpha,\beta}\langle \xi\rangle^{-m-\rho|\alpha|+\delta|\beta|}, \,\,|\alpha|\leq \varkappa,|\beta|\leq l,\,\,m\geq \varkappa(1-\rho)|\frac{1}{p}-\frac{1}{2}|+\delta l,$
\end{itemize}
the corresponding pseudo-differential operator $A$ extends to a bounded operator from $B^r_{p,q}(G)$ into $B^r_{p,q}(G)$ for all $r\in \mathbb{R},$ $1<p<\infty$ and $0<q\leq \infty,$ (see Cardona \cite{Cardona,Cardona2,Cardona22}). For $p,q=\infty$ we have the following theorem  which is our main result in this paper.

\begin{theorem}\label{MainT} Let $G$ be a compact Lie group of dimension $n$. Let $0<\rho\leq 1,$ $k:=[\frac{n}{2}]+1,$  and let $A:C^{\infty}(G)\rightarrow C^{\infty}(G)$ be a pseudo-differential operator with symbol $\sigma$ satisfying  
\begin{equation}\label{eqMaint}
\Vert \mathbb{D}_{\xi}^{\alpha}\sigma(x,\xi)\Vert_{op}\leq C_{\alpha}\langle \xi \rangle^{-\frac{n}{2}(1-\rho)-\rho|\alpha|}
\end{equation}
 for all $|\alpha| \leq k.$ Then  $A:B^s_{\infty,\infty}(G)\rightarrow B^s_{\infty,\infty}(G)$ extends to a bounded linear operator for all $-\infty<s<\infty.$ Moreover,\begin{equation}
\Vert A\Vert_{\mathscr{B}(B^s_{\infty,\infty})}\leq C  \sup\{C_{\alpha}: {|\alpha|\leq k}\}.\end{equation}
\end{theorem}
 Theorem \ref{MainT} implies the following result.
 \begin{corollary}\label{MainT'} Let $G$ be a compact Lie group of dimension $n$. Let $0<\rho\leq 1,$ $0\leq \delta\leq 1,$ $\ell\in\mathbb{N},$ $k:=[\frac{n}{2}]+1,$  and let $A:C^{\infty}(G)\rightarrow C^{\infty}(G)$ be a pseudo-differential operator with symbol $\sigma$ satisfying  
\begin{equation}\label{eqMaint'}
\Vert \partial_x^\beta\mathbb{D}_{\xi}^{\alpha}\sigma(x,\xi)\Vert_{op}\leq C_{\alpha}\langle \xi \rangle^{-m-\rho|\alpha|+\delta|\beta|}
\end{equation}
 for all $|\alpha| \leq k,$ $|\beta|\leq \ell.$ Then  $A:B^s_{\infty,\infty}(G)\rightarrow B^s_{\infty,\infty}(G)$ extends to a bounded linear operator for all $-\infty<s<\infty$ provided that $m\geq \delta \ell+\frac{n}{2}(1-\rho).$
\end{corollary}

 Besov spaces on graded Lie groups, as well as the action of Fourier multipliers and spectral multipliers on these spaces can be found in Cardona and Ruzhansky \cite{Cardona33,Cardona333}. We refer the reader to the references \cite{DuSpe,Sha} for boundedness properties of pseudo-differential operators in Besov spaces on $\mathbb{R}^n.$
 This paper is organized as follows. In the next section we present some basics on the calculus of global operators. Finally, in Section \ref{proof}  we prove our Besov estimates.

\section{Pseudo-differential operators on compact Lie groups}\label{Preliminaries}
\subsection{Fourier analysis and Sobolev spaces on compact Lie groups } In this section we will introduce some preliminaries on pseudo-differential operators on compact Lie groups and some of its properties on $L^p$-spaces. There are two notions of pseudo-differential operators on compact Lie groups. The first notion in the case of general manifolds (based on the idea of {\em local symbols} as in H\"ormander \cite{Ho-1}) and, in a much more recent context, the one of global pseudo-differential operators on compact Lie groups as defined by Ruzhansky and Turunen \cite{Ruz}. We adopt this last notion for our work, because we will use a description of the Besov spaces $B^s_\infty,\infty$ trough representation theory. We will always equip a compact Lie group with the Haar measure $\mu_{G}.$ For simplicity, we will write  $L^\infty(G)$ for $L^\infty(G,\mu_{G})$.  The following assumptions are respectively the Fourier transform and the Fourier inversion formula for smooth functions,
 $$  \widehat{\varphi}(\xi)=\int_{G}\varphi(x)\xi(x)^*dx,\,\,\,\,\,\,\,\,\,\,\,\,\,\,\,\, \varphi(x)=\sum_{[\xi]\in \widehat{G}}d_{\xi}\text{Tr}(\xi(x)\widehat{\varphi}(\xi)) .$$
The Peter-Weyl Theorem on $G$ implies  the Plancherel identity on $L^2(G),$
$$ \Vert \varphi \Vert_{L^2(G)}= \left(\sum_{[\xi]\in \widehat{G}}d_{\xi}\text{Tr}(\widehat{\varphi}(\xi)\widehat{\varphi}(\xi)^*) \right)^{\frac{1}{2}}=\Vert  \widehat{\varphi}\Vert_{ L^2(\widehat{G} ) } .$$
\noindent Here  $$\Vert A \Vert^2_{HS}=\text{Tr}(AA^*),$$ denotes the Hilbert-Schmidt norm of matrices. Now, we introduce global pseudo-differential operstors in the sense of Ruzhansky and Turunen. Any continuous linear operator $A$ on $G$ mapping $C^{\infty}(G)$ into $\mathcal{D}'(G)$ gives rise to a {\em matrix-valued global (or full) symbol} $\sigma_{A}(x,\xi)\in \mathbb{C}^{d_\xi \times d_\xi}$ given by
\begin{equation}
\sigma_A(x,\xi)=\xi(x)^{*}(A\xi)(x),
\end{equation}
which can be understood from the distributional viewpoint. Then it can be shown that the operator $A$ can be expressed in terms of such a symbol as \cite{Ruz}
\begin{equation}Af(x)=\sum_{[\xi]\in \widehat{G}}d_{\xi}\text{Tr}[\xi(x)\sigma_A(x,\xi)\widehat{f}(\xi)]. 
\end{equation}
In this paper we use the notation $\text{Op}(\sigma_A)=A.$ $L^p(\widehat{G})$ spaces on the unitary dual can be well defined.   If $p=2,$ $L^2(\widehat{G})$ is defined by the norm  $$\Vert \Gamma \Vert^2_{L^2(\widehat{G})}=\sum_{[\xi]\in\widehat{G}}d_{\xi}\Vert \Gamma (\xi)\Vert^2_{HS}.$$ 
Now, we want to introduce Sobolev spaces and, for this, we give some basic tools. \noindent Let $\xi\in \textnormal{Rep}(G):=\cup \widehat{G}=\{\xi:[\xi]\in\widehat{G}\},$ if $x\in G$ is fixed, $\xi(x):H_{\xi}\rightarrow H_{\xi}$ is an unitary operator and $d_{\xi}:=\dim H_{\xi} <\infty.$ There exists a non-negative real number $\lambda_{[\xi]}$ depending only on the equivalence class $[\xi]\in \hat{G},$ but not on the representation $\xi,$ such that $-\mathcal{L}_{G}\xi(x)=\lambda_{[\xi]}\xi(x);$ here $\mathcal{L}_{G}$ is the Laplacian on the group $G$ (in this case, defined as the Casimir element on $G$). Let  $\langle \xi\rangle$ denote the function $\langle \xi \rangle=(1+\lambda_{[\xi]})^{\frac{1}{2}}$.  
\begin{definition}\label{sov} For every $s\in\mathbb{R},$ the {\em Sobolev space} $H^s(G)$ on the Lie group $G$ is  defined by the condition: $f\in H^s(G)$ if only if $\langle \xi \rangle^s\widehat{f}\in L^{2}(\widehat{G})$. The Sobolev space $H^{s}(G)$ is a Hilbert space endowed with the inner product $\langle f,g\rangle_{s}=\langle \Lambda_{s}f, \Lambda_{s}g\rangle_{L^{2}(G)}$, where, for every $r\in\mathbb{R}$, $\Lambda_{s}:H^r\rightarrow H^{r-s}$ is the bounded pseudo-differential operator with symbol $\langle \xi\rangle^{s}I_{\xi}.$ In $L^p$ spaces,  the $p$-Sobolev space of order $s,$ $H^{s,p}(G),$ is defined by functions satisfying
\begin{equation}
\Vert f \Vert_{H^{s,p}(G)}:=\Vert \Lambda_sf\Vert_{L^p(G)}<\infty.
\end{equation}
\end{definition}
\subsection{Differential and difference operators} In order to classify symbols by its regularity we present the usual definition of differential operators and the difference operators used introduced in \cite{Ruz}.
\begin{definition}\label{differenceoperators} Let $(Y_{j})_{j=1}^{\text{dim}(G)}$ be a basis for the Lie algebra $\mathfrak{g}$ of $G$, and let $\partial_{j}$ be the left-invariant vector fields corresponding to $Y_j$. We define the differential operator associated to such a basis by $D_{Y_j} = \partial_{j}$ and, for every $\alpha\in\mathbb{N}^n$, the {\em differential operator} $\partial^{\alpha}_{x}$ is the one given by $\partial_x^{\alpha}=\partial_{1}^{\alpha_1}\cdots \partial_{n}^{\alpha_n}$. Now, if $\xi_{0}$ is a fixed irreducible  representation, the matrix-valued {\em difference operator} is the given by $\mathbb{D}_{\xi_0}=(\mathbb{D}_{\xi_0,i,j})_{i,j=1}^{d_{\xi_0}}=\xi_{0}(\cdot)-I_{d_{\xi_0}}$. If the representation is fixed we omit the index $\xi_0$ so that, from a sequence $\mathbb{D}_1=\mathbb{D}_{\xi_0,j_1,i_1},\cdots, \mathbb{D}_n=\mathbb{D}_{\xi_0,j_n,i_n}$ of operators of this type we define $\mathbb{D}^{\alpha}_{\xi}=\mathbb{D}_{1}^{\alpha_1}\cdots \mathbb{D}^{\alpha_n}_n$, where $\alpha\in\mathbb{N}^n$.
\end{definition}
\subsection{Besov spaces} We introduce the Besov spaces on compact Lie groups using the  Fourier transform on the group $G$ as follow.
\begin{definition}\label{besovspaces} Let $r\in\mathbb{R},$ $0\leq q<\infty$ and $0<p\leq \infty.$ If $f$ is a measurable function on $G,$ we say that $f\in B^r_{p,q}(G)$ if $f$ satisfies
\begin{equation}\label{n1}\Vert f \Vert_{B^r_{p,q}}:=\left( \sum_{m=0}^{\infty} 2^{mrq}\Vert \sum_{2^m\leq \langle \xi\rangle< 2^{m+1}}  d_{\xi}\text{Tr}[\xi(x)\widehat{f}(\xi)]\Vert^q_{L^p(G)}\right)^{\frac{1}{q}}<\infty.
\end{equation}
If $q=\infty,$ $B^r_{p,\infty}(G)$ consists of those functions $f$ satisfying
\begin{equation}\label{n2}\Vert f \Vert_{B^r_{p,\infty}}:=\sup_{m\in\mathbb{N}} 2^{mr}\Vert \sum_{2^m\leq \langle \xi\rangle < 2^{m+1}}  d_{\xi}\text{Tr}[\xi(x)\widehat{f}(\xi)]\Vert_{L^p(G)}<\infty.
\end{equation}
\end{definition}
If we denote by $\textnormal{Op}(\chi_{m})$ the Fourier multiplier associated to the symbol
$$ \chi_{m}(\eta)=1_{\{[\xi]:2^m\leq \langle \xi\rangle< 2^{m+1}\}}(\eta), $$ we also write,
\begin{equation}
\Vert f\Vert_{B^r_{p,q}}=\Vert \{2^{mr} \Vert \textnormal{Op}(\chi_{m})f \Vert_{L^p(G)}\}_{m=0}^{\infty} \Vert_{l^q(\mathbb{N})},\,\,0<p,q\leq \infty,\,r\in\mathbb{R}.
\end{equation}
\begin{remark}
For every $s\in\mathbb{R},$ $H^{s,2}(G)=H^{s}(G)=B^{s}_{2,2}(G).$ Besov spaces according to Definition \eqref{besovspaces} were introduced in \cite{RuzBesov} on compact homogeneous manifolds where, in particular, the authors obtained its embedding properties. On compact Lie groups such spaces were characterized, via representation theory, in \cite{RuzBesov2}.
\end{remark}

\section{Global operators on $B^s_{\infty,\infty}(G)$-spaces}\label{proof}
In this section we prove our Besov estimate for global pseudo-differential operators. Our starting point is the following lemma which is slight variation of one due to J. Delgado and M. Ruzhansky (see Lemma 4.11 of \cite{DRlp}) and whose proof is verbatim to Delgado-Ruzhansky's proof.

\begin{lemma}\label{lemma} Let $G$ be a compact Lie group of dimension $n$. Let $0<\rho\leq 1,$ $k:=[\frac{n}{2}]+1,$  and let $A:C^{\infty}(G)\rightarrow C^{\infty}(G)$ be a pseudo-differential operator with symbol $\sigma$ satisfying  
\begin{equation}
\Vert \mathbb{D}_{\xi}^{\alpha}\sigma(x,\xi)\Vert_{op}\leq C_{\alpha}\langle \xi \rangle^{-\frac{n}{2}(1-\rho)-\rho|\alpha|}
\end{equation}
 for all $|\alpha| \leq k.$ 
 Let us assume that $\sigma$ is supported in  $\{\xi:R\leq \langle\xi\rangle\leq 2R\}$ for some $R>0.$ Then  $A:L^{\infty}(G)\rightarrow L^{\infty}(G)$ extends to a bounded linear operator with norm operator independent of $R.$ Moreover,\begin{equation}
\Vert A\Vert_{\mathscr{B}(L^{\infty})}\leq C  \sup\{C_{\alpha}: {|\alpha|\leq k}\}.\end{equation}
\end{lemma} So, we are ready for the proof of our main result.
\begin{proof}[Proof of Theorem \ref{MainT}] Our proof consists of two steps. In the first one, we prove the statement of the theorem for Fourier multipliers, i.e., pseudo-differential operators depending only on the Fourier variables $\xi.$ Later,  we extend the result for general global  operators.\\

\noindent{{\textit{Step 1.}}} Let us consider $\mathcal{R}:=(I-\mathcal{L}_G)^{\frac{1}{2}}$ where $\mathcal{L}_{G}$ is the Laplacian on $G.$ Let us denote by $\psi$  the characteristic function of the interval $I=[1/2,1].$ Denote by $\psi_{l}$ the function $\psi_{l}(t)=\psi(2^{-l}t),$ $t\in \mathbb{R}.$ 
We will use the following characterization for $B^s_{\infty,\infty}(G):$  $f\in B^s_{\infty,\infty}(G),$ if  and only if, 
\begin{equation}
\Vert f \Vert_{B^s_{\infty,\infty}(G)}:=\sup_{l\geq 0}2^{ls}\Vert \psi_l(\mathcal{R})f \Vert_{L^\infty(G)}
\end{equation}
where $\psi_l(\mathcal{R})$ is defined by the functional calculus associated to the self-adjoint operator $\mathcal{R}.$ If $A\equiv\sigma(D_x  )$ has a symbol depending only on the Fourier variables $\xi$, then 
\begin{equation}
\Vert    \sigma(D_x)f\Vert_{B^s_{\infty,\infty}(G)}:=\sup_{l\geq 0}2^{ls}\Vert \psi_l(\mathcal{R})\sigma(D_x)f \Vert_{L^\infty(G)}.
\end{equation}
Taking into account that the operator $\sigma(D_x)$ commutes with $\psi_l(\mathcal{R})$ for every $l,$ that 
\begin{equation}
\psi_l(\mathcal{R})\sigma(D_x)= \sigma(D_x)\psi_l(\mathcal{R})= \sigma_{l}(D_x)\psi_l(\mathcal{R})
\end{equation}
where $\sigma_l(D_x)$ is the pseudo-differential operator with matrix-valued symbol $$\sigma_{l}(\xi)=\sigma(\xi) \cdot 1_{\{\xi:2^{l-1}\leq \langle \xi\rangle\leq 2^{l+1}   \}},$$
and that $\sigma_{l}(D_x)$ has a symbol supported in $\{\xi:2^{l-1}\leq \langle \xi\rangle\leq 2^{l+1}   \},$ by Lemma \ref{lemma} we deduce that $\sigma_{l}(D_x)$ is a bounded operator on $L^{\infty}(G)$ with operator norm independent on $l.$ In fact,  $\sigma_{l}$ satisfies the symbol inequalities
$$\Vert \mathbb{D}_{\xi}^{\alpha}\sigma_l(\xi)\Vert_{op}\leq C_{\alpha}\langle \xi \rangle^{-\frac{n}{2}(1-\rho)-\varepsilon|\alpha|},$$ for all $|\alpha| \leq k,$ and consequently
\begin{equation}
\Vert \sigma_{l}(D_x)\Vert_{\mathscr{B}(L^{\infty})}\leq C \sup\{C_{\alpha}: {|\alpha|\leq k}\}.\end{equation} So, we have
\begin{align*}
\Vert \psi_l(\mathcal{R})\sigma(D_x) & f \Vert_{L^\infty(G)} \\
&=\Vert   \sigma_l(D_x) \psi_l(\mathcal{R})f \Vert_{L^\infty(G)}\leq \Vert \sigma_{l}(D_x)\Vert_{\mathscr{B}(L^{\infty})} \Vert    \psi_l(\mathcal{R})f \Vert_{L^\infty(G)} \\
& \lesssim \sup\{C_{\alpha}: {|\alpha|\leq k}\}\Vert    \psi_l(\mathcal{R})f \Vert_{L^\infty(G)}.
\end{align*}
As a consequence, we obtain
\begin{align*}\Vert \sigma(D_x) f \Vert_{B^s_{\infty,\infty}(G)} &= \sup_{l\geq 0}2^{ls}\Vert \psi_l(\mathcal{R})\sigma(D_x)f \Vert_{L^\infty(G)}\\
& \lesssim \sup\{C_{\alpha}: {|\alpha|\leq k}\}\sup_{l\geq 0} 2^{ls} \Vert    \psi_l(\mathcal{R})f \Vert_{L^\infty(G)}\\
&\asymp \sup\{C_{\alpha}: {|\alpha|\leq k}\}\Vert f\Vert_{B^s_{\infty,\infty }(G)}. 
\end{align*}
\noindent{{\textit{Step 2.}}} Now, we extend the estimate  from multipliers to pseudo-differential operators. So,  let us define for every $z\in G,$ the multiplier
\begin{center}
$ \sigma_z(D_{x})f(x)=\sum_{[\xi]\in \widehat{G}}d_{\xi}\text{Tr}[\xi(x)\sigma(z,\xi)\widehat{f}(\xi)]. $
\end{center}
For every $x\in G$ we have the equality,
\begin{center}
$ \sigma_{x}(D_{x})f(x)=Af(x),$ 
\end{center} 
and we can estimate the Besov norm of the function $\sigma(x,D_x)f,$ as follows
\begin{align*}
\Vert \sigma_{x}(D_x)f(x)\Vert_{B^s_{\infty,\infty }} &\asymp \sup_{l\geq 0}2^{ls} \textnormal{ess sup}_{x\in G}|\psi_l(\mathcal{R})\sigma(x,D_x)f(x) |\\
&= \sup_{l\geq 0}2^{ls} \textnormal{ess sup}_{x\in G}|\sum_{[\xi]\in \widehat{G}}d_{\xi}\text{Tr}[\xi(x)\psi_l(\xi)\sigma(x,\xi)\widehat{f}(\xi)] |\\
&\leq \sup_{l\geq 0}2^{ls} \textnormal{ess sup}_{x\in G}\sup_{z\in {G}}|\sum_{[\xi]\in \widehat{G}}d_{\xi}\text{Tr}[\xi(x)\psi_l(\xi)\sigma(z,\xi)\widehat{f}(\xi)] |\\
&=  \sup_{l\geq 0}2^{ls} \textnormal{ess sup}_{x\in G}   \sup_{z\in G}|\psi_l(\mathcal{R})\sigma_z(D_x)f(x) | \\
&=  \sup_{l\geq 0}2^{ls} \textnormal{ess sup}_{x\in G}   \sup_{z\in G}| \sigma_z(D_x)  \psi_l(\mathcal{R})f(x) | \\
&\leq \sup_{l\geq 0}2^{ls}  \textnormal{ess sup}_{x\in G}   \sup_{z\in G} \textnormal{ess sup}_{\varkappa\in G} |\sigma_z(D_\varkappa)  \psi_l(\mathcal{R})f(\varkappa) | \\
&= \sup_{l\geq 0}2^{ls}   \sup_{z\in G} \Vert \sigma_z(D_\varkappa)  \psi_l(\mathcal{R})f(\varkappa) \Vert_{L^\infty(G)} .\\
\end{align*} From the estimate for the operator norm of multipliers proved in the first step, we deduce
\begin{align*}\sup_{z\in G}\Vert \sigma_{z}(D_x) \psi_l(\mathcal{R}) f \Vert_{L^\infty(G)}\lesssim  \sup\{C_{\alpha}: {|\alpha|\leq k}\}\Vert \psi_l(\mathcal{R})f\Vert_{L^\infty(G)}. 
\end{align*} So, we have
\begin{equation}
\Vert \sigma_{x}(D_x)f(x)\Vert_{B^s_{\infty,\infty }} \lesssim \sup\{C_{\alpha}: {|\alpha|\leq k}\}\Vert f\Vert_{B^s_{\infty,\infty }
(G)}. 
\end{equation} Thus, we finish the proof.
\end{proof} Now, we present the following result for symbols admitting differentiability in the spatial variables.
\begin{corollary}\label{MainT2} Let $G$ be a compact Lie group of dimension $n$. Let $0<\rho\leq 1,$ $0\leq \delta\leq 1,$ $\ell\in\mathbb{N},$ $k:=[\frac{n}{2}]+1,$  and let $A:C^{\infty}(G)\rightarrow C^{\infty}(G)$ be a pseudo-differential operator with symbol $\sigma$ satisfying  
\begin{equation}
\Vert \partial_x^\beta\mathbb{D}_{\xi}^{\alpha}\sigma(x,\xi)\Vert_{op}\leq C_{\alpha}\langle \xi \rangle^{-m-\rho|\alpha|+\delta|\beta|}
\end{equation}
 for all $|\alpha| \leq k,$ $|\beta|\leq \ell.$ Then  $A:B^s_{\infty,\infty}(G)\rightarrow B^s_{\infty,\infty}(G)$ extends to a bounded linear operator for all $-\infty<s<\infty$ provided that $m\geq \delta \ell+\frac{n}{2}(1-\rho).$
\end{corollary} \begin{proof}
Let us  observe that 
\begin{equation}
\langle \xi \rangle^{-m-\rho|\alpha|+\delta|\beta|}\leq \langle \xi \rangle^{-\frac{n}{2}(1-\rho)+\rho|\alpha|}
\end{equation} when $m\geq \delta \ell+\frac{n}{2}(1-\rho).$ So, we finish the proof if we apply Theorem \ref{MainT}.
\end{proof}

\noindent \textbf{Acknowledgments:}  This project was partially supported by the  Department of Mathematics of the  Pontificia Universidad Javeriana,  Bogot\'a-Colombia.

\bibliographystyle{amsplain}

\end{document}